\documentclass[12pt]{amsart}

\usepackage[margin=1in]{geometry}
\usepackage[colorlinks=true,linkcolor=blue,citecolor=blue]{hyperref}
\usepackage{cite}

\usepackage{amsmath,amssymb,amsthm,amsfonts,mathtools}
\usepackage{hyperref}
\usepackage[T1]{fontenc}
\usepackage{xcolor}

\theoremstyle{plain}
\newtheorem{theorem}{Theorem}[section]
\newtheorem{lemma}[theorem]{Lemma}
\newtheorem{proposition}[theorem]{Proposition}

\theoremstyle{definition}

\theoremstyle{remark}
\newtheorem{remark}[theorem]{Remark}

\numberwithin{equation}{section}
\newcommand{\norm}[1]{\lVert#1\rVert}

\DeclareMathOperator{\Real}{Re}
\newcommand{\D}{\mathbb{D}}
\newcommand{\T}{\mathbb{T}}

\newcommand{\B}{\mathcal{B}}
\newcommand{\Hp}{\mathcal{H}}

\begin{document}
	\title[Norm of the generalized Hilbert operator  ]{Norm of the generalized Hilbert operator on Hardy spaces}
	
\author{  Songxiao Li, Weiye Pan$^\ast$  and Mengmeng Zhou }

\address{Songxiao Li\\   Department of Mathematics, Shantou University, Shantou 515063, Guangdong, P.R. China.  }
\email{jyulsx@163.com}

\address{ Weiye Pan  \\ Department of Mathematics, Sichuan University,  Chengdu, Sichuan 610065, P.R. China.}
	\email{pan\_weiye@scu.edu.cn  }

\address{Mengmeng Zhou\\   Department of Mathematics, Shantou University, Shantou 515063, Guangdong, P.R. China.  }
\email{25mmzhou@stu.edu.cn}

	
	\begin{abstract}   We study the generalized Hilbert operator
\[
\mathcal{H}_b f(z)=\int_0^1 f(t)\,\frac{(1-t)^b}{(1-tz)^{b+1}}\,dt,
\qquad b>0,
\]
acting on the Hardy spaces $H^p$ for $1\leq p\leq \infty$.
We establish the precise operator norm
\[
\|\mathcal{H}_b\|_{H^p\to H^p}=B\!\left(\frac1p,b+1-\frac1p\right)
\]
for every $1<p<\infty$ and, by continuous extension to $b=0$, recover the classical norm $\pi/\sin(\pi/p)$.
We also prove that $\mathcal{H}_b$ is bounded on $H^1$ for every $b>0$, in contrast with the classical Hilbert operator, and we obtain the sharp restricted norm estimate
\[
\|\mathcal{H}_b\|_{H^1_0\to H^1}=B(1,b).
\]
We also determine the exact norm
\[
\|\mathcal{H}_b\|_{H^\infty\to \mathcal B}=\frac{1}{b+1}+2.
\]

	\thanks{$^\ast$ Corresponding author.}
	\end{abstract}
		\maketitle

\section{Introduction} \vskip 2mm

\label{sec:intro}

Let $\mathbb D=\{z\in\mathbb C: |z|<1\}$ be the unit disc and let $H(\mathbb D)$ be the space of all analytic functions on $\mathbb D$.
For $0<p\le\infty$ the Hardy space $H^p$ consists of those $f\in H(\mathbb D)$ for which
\[
\|f\|_{H^p}^p:=\sup_{0\le r<1}\frac{1}{2\pi}\int_0^{2\pi}|f(re^{i\theta})|^p\,d\theta<\infty,
\qquad 0<p<\infty,
\]
and $\|f\|_{H^\infty}=\sup_{z\in\mathbb D}|f(z)|$.
We denote by $H^1_0$ the subspace of $H^1$ consisting of functions vanishing at the origin:
\[
H^1_0:=\{f\in H^1: f(0)=0\}.
\]
 Recall that the Bloch space consists of all analytic functions $f$ on $\D$ for which
\[
\|f\|_\B:=|f(0)|+\sup_{z\in\D}(1-|z|^2)|f'(z)|<\infty.
\]
It is well known that $H^\infty \subset \B$. We refer to \cite{Duren70, zkh} for the standard theory of these spaces.

The classical Hilbert matrix $\mathcal{H}=\left(\frac{1}{n+k+1}\right)_{n,k\ge 0}$ induces an
operator on $H(\D)$ by its action on Taylor coefficients: for
$f(z)=\sum_{k=0}^{\infty}a_{k}z^{k}$,
\begin{equation}%\label{eq:H-classical}
\mathcal{H} f(z)=\sum_{n=0}^{\infty}
\Bigl(\sum_{k=0}^{\infty}\frac{a_{k}}{n+k+1}\Bigr)z^{n}. \nonumber
\end{equation}
Equivalently, $\mathcal{H}$ admits the integral representation~\cite{DiamantopoulosSiskakis2000}
\begin{equation}%\label{eq:H-int}
\mathcal{H} f(z)=\int_{0}^{1}\frac{f(t)}{1-tz}\,dt,\qquad z\in\D. \nonumber
\end{equation}
The convergence of the integral is guaranteed for all $f\in H^{1}$
by the Fej\'er--Riesz inequality~\cite[Theorem~3.13]{Duren70} .

Diamantopoulos and Siskakis~\cite{DiamantopoulosSiskakis2000} proved that $\mathcal{H}$ is bounded
on $H^{p}$ for $1<p<\infty$ and obtained the norm estimate
$\norm{\mathcal{H}}_{H^{p}\to H^{p}}\le\pi/\sin(\pi/p)$ for $2\le p<\infty$.
Subsequently, based on a Nehari-type theorem, Dostani\'c, Jevti\'c, and Vukoti\'c~\cite{DostanicJevticVukotic2008}
established the exact norm
\begin{equation}%\label{eq:H-norm}
\norm{\mathcal{H}}_{H^{p}\to H^{p}}= \frac{\pi}{\sin(\pi/p)} ,\qquad 1<p<\infty.\nonumber
\end{equation}
It is well known that $\mathcal{H}$ fails to be bounded on both $H^{1}$
and $H^{\infty}$~\cite{DiamantopoulosSiskakis2000}.

In this paper we consider the generalized Hilbert operator
\begin{equation}\label{eq:Hb}
\mathcal{H}_b f(z)=\int_0^1 f(t)\,\frac{(1-t)^b}{(1-tz)^{b+1}}\,dt,
\qquad b>0,\ z\in\D,
\end{equation}
acting on the scale $H^p$, $1\leq p\leq \infty$. The generalized Hilbert operator $\Hp_{b}$ was introduced by Li and Stevi\'c~\cite{LiStevic2009}. See   \cite{bw, Li2009} for results of $\Hp_{b}$ on Dirichlet-type spaces.  Bellavita, Chalmoukis, Daskalogiannis, and
Stylogiannis~\cite{Bellavita2024} developed a general framework
for operators arising from Hausdorff matrices, which includes
$\Hp_{b}$ as a special case.

The motivation for studying $\Hp_{b}$ is threefold.  First, the parameter $b$ provides a natural one-parameter deformation of the
classical Hilbert operator, and understanding how the norm depends on $b$ sheds light on the structural stability of Hilbert-type
operators.  Second, the additional factor $(1-t)^{b}$ acts as a regularizing term near the endpoint $t=1$, which restores
boundedness on $H^{1}$, a property that the classical operator lacks.  Third, the generalized operator serves as a prototype for
the broader class of Hausdorff operators, and the techniques developed here, can be extend to other integral type operators.

The purpose of this paper is to determine the exact norm of $\mathcal H_b$ on Hardy spaces $H^p$ for $1\le p<\infty$, to obtain a sharp restricted norm estimate from $H^1_0$ to $H^1$, and to establish its norm from $H^\infty$ into the Bloch space $\mathcal B$. Our main results are the following
three theorems.

\begin{theorem}\label{thm:main}
Let $b>0$ and $1<p<\infty$.  Then
\begin{equation}\label{eq:main}
\|\mathcal{H}_b\|_{H^p\to H^p}=B\Bigl(\frac1p,b+1-\frac1p\Bigr).
\end{equation}
\end{theorem}
Since $B(s,1-s)=\pi/\sin(\pi s)$, putting $b=0$ in \eqref{eq:main} gives $B(1/p,1-1/p)=\pi/\sin(\pi/p)$, which is the classical value.

We also consider the endpoint spaces $H^1$ and $H^\infty$.

\begin{theorem}\label{thm:endpoints}
Let $b>0$.
\begin{enumerate}
\item[{\bf (i)}] $\mathcal{H}_b$ is bounded on $H^1$, and
\begin{equation}\label{eq:main1}
\|\mathcal{H}_b\|_{H^1_0\to H^1}=B\Bigl(1,b\Bigr).
\end{equation}
\item[{\bf (ii)}] $\mathcal{H}_b$ is unbounded on $H^\infty$.
\end{enumerate}
\end{theorem}

From Theorem \ref{thm:endpoints}, we see that the operator $\mathcal{H}_b:H^\infty\to H^\infty$ is unbounded. If we slightly enlarge the target space, for example the Bloch space, the boundedness can be established.  Finally, we compute the exact norm of $\mathcal{H}_b:H^\infty\to \B$.

\begin{theorem}\label{thm:bloch}
Let $b>0$.  Then
\begin{equation}\label{eq:bloch}
\|\mathcal{H}_b\|_{H^\infty\to\B}=\frac{1}{b+1}+2.
\end{equation}
\end{theorem}
Our contribution advances this line of research by determining the \emph{exact} norm of $\mathcal H_b$ on $H^p$ for the range $1<p<\infty$, by obtaining a sharp restricted norm estimate from $H^1_0$ to $H^1$, and by establishing $\|\mathcal H_b\|_{H^\infty\to\mathcal B}$, all of which were not previously known.

The proofs of Theorems \ref{thm:main}--\ref{thm:bloch} are given in the sections below.  The main technical tool is the representation of $\mathcal{H}_b$ as an average of weighted composition operators, together with a sharp change-of-variables formula (Lemma \ref{lem:wco}) and Gabriel's classical theorem (Lemma \ref{lem:gabriel}). Throughout the paper, $C$ denotes a positive absolute constant, the value of which may vary from line to line.

\smallskip

\section{Preliminaries}
\label{sec:prelim}

 In this section, we state some necessary definitions and auxiliary lemmas, which will be used in this paper.

\subsection{Weighted composition representation}

For $b>0$ and $0<t<1$,  define the symbol
\begin{equation}\label{eq:phit}
\phi_t(z)=\frac{t}{1-(1-t)z},\qquad z\in\D,
\end{equation}
and the weight
\begin{equation}%\label{eq:psi}
w_t(z)=\frac{(1-t)^b}{1-(1-t)z},\qquad z\in\D.\nonumber
\end{equation}
The associated weighted composition operator is
\[
T_t f(z)=w_t(z)\,f(\phi_t(z)),\qquad f\in H(\D).
\]
A direct change of variables in \eqref{eq:Hb} shows that
\begin{equation}\label{eq:WCO}
\mathcal{H}_b f(z)=\int_0^1\,T_t f(z)\,dt,
\qquad z\in\D.
\end{equation}
Indeed, substituting $x=\phi_t(z)$ in the integral in \eqref{eq:WCO} gives
$t=x(1-z)/(1-xz)$, $1-t=(1-x)/(1-xz)$, $w_t(z)=(1-xz)/(1-z)$ and
$dt=(1-z)/(1-xz)^2\,dx$, whence
\[
\int_0^1w_t(z)\,f(\phi_t(z))\,dt=\int_0^1\frac{(1-x)^b}{(1-xz)^{b+1}}\,f(x)\,dx.
\]
The map $\phi_t$ sends $\D$ onto the disc $D(a_t,r_t)$ with centre $a_t=1/(2-t)$ and radius $r_t=(1-t)/(2-t)$; in particular $\phi_t(\D)\subset\D$ and $\phi_t(1)=1$, so that $\phi_t(\T)$ is the circle tangent to $\T$ at the point $1$.

\subsection{Special functions}

The Beta function is
\[
B(s,t)=\int_0^1 x^{s-1}(1-x)^{t-1}\,dx
=\frac{\Gamma(s)\Gamma(t)}{\Gamma(s+t)},
\qquad \Real s,\,\Real t>0.
\]
It satisfies $B(s,t)=B(t,s)$ and $B(s,1-s)=\pi/\sin(\pi s)$ for $0<s<1$.

The Gaussian hypergeometric function is
\[
F(a,b,c;z)=\sum_{k=0}^\infty\frac{(a)_k(b)_k}{(c)_k}\frac{z^k}{k!},
\qquad |z|<1,
\]
where $(a)_k=\Gamma(a+k)/\Gamma(a)$ is the Pochhammer symbol.  We shall use the Euler integral representation
\begin{equation}\label{eq:Euler}
F(a,b,c;z)=\frac{1}{B(a,c-a)}\int_0^1 x^{a-1}(1-x)^{c-a-1}(1-xz)^{-b}\,dx,
\end{equation}
valid for $\Real c>\Real a>0$.

\subsection{Hardy Spaces}

For $1\le p<\infty$,  every $f\in H^p$ has radial boundary values $f^*(\zeta)$ for a.e.~$\zeta\in\T:= \partial \mathbb{D}$, and
\begin{equation}\label{eq:boundary}
\|f\|_{H^p}^p=\frac{1}{2\pi}\int_\T |f^*(\zeta)|^p\,|d\zeta|.
\end{equation}
When no confusion arises, we shall denote both $f$ and its boundary function $f^*$ by the same symbol $f$.

For $1<p<\infty$ the dual space $(H^p)^*$ is isomorphic to $H^q$, where $1/p+1/q=1$, under the Cauchy pairing
\begin{equation}\label{eq:pairing}
\langle f,g\rangle=\frac{1}{2\pi}\int_\T f(\zeta)\,\overline{g(\zeta)}\,|d\zeta|,
\qquad f\in H^p,\ g\in H^q.
\end{equation}
Consequently, if $T:H^p\to H^p$ is bounded, then its adjoint $T^*:H^q\to H^q$ satisfies $$\|T^*\|_{H^q\to H^q}=\|T\|_{H^p\to H^p}.$$

We also need the Fej\'er--Riesz inequality: there is an absolute constant $C>0$ such that
\begin{equation}\label{eq:FR}
\int_0^1 |f(x)|\,dx\le C\,\|f\|_{H^1}
\end{equation}
for every $f\in H^1$; see Duren \cite[Theorem~3.13]{Duren70}.

\subsection{Some lemmas}

The next lemma is the key technical identity.  It gives the exact $H^p$-norm of the weighted composition operator $T_t$ as an integral over the inner circle $\phi_t(\T)$.

\begin{lemma}\label{lem:wco}
Let $1\le p<\infty$, $f\in H^p$ and $0<t<1$.  Then
\begin{equation}%\label{eq:wco}
\|T_t f\|_{H^p}
=\frac{t^{\frac{1}{p}-1}}{(1-t)^{\frac{1}{p}-b}} \left(\frac{1}{2\pi} \int_{\phi_t(\T)} |f(w)|^p |w|^{p-2} |dw| \right)^{\frac{1}{p}}.\nonumber
\end{equation}
\end{lemma}

\begin{proof}
Differentiating both sides of \eqref{eq:phit} with respect to $z$ and taking absolute values, we obtain
 \begin{align} %\label{eqp}
  |\phi_t'(z)|=\frac{t}{(1-t)^{2b-1}}|w_t(z)|^2,\,\,\,0 < t < 1,\,\,\,z \in \mathbb{D}.\nonumber
  \end{align}
 Applying a change of variables, we get
\begin{align*}
\|T_t f\|_{H^p}
&=  \left(\frac{1}{2\pi}  \int_{\mathbb{T}} |f(\phi_t(\zeta))|^p |w_t(\zeta)|^p  |d\zeta| \right)^{\frac{1}{p}} \notag \\
&=\left[\frac{1}{2\pi}\frac{1}{ t(1-t)^{1-2b}} \int_{\mathbb{T}} |f(\phi_t(\zeta))|^p |w_t(\zeta)|^{p-2} \frac{t}{(1-t)^{2b-1}}|w_t(z)|^2|d\zeta|\right]^{\frac{1}{p}} \notag \\
&= \frac{t^{-\frac{1}{p}}}{(1-t)^{\frac{1}{p}-\frac{2b}{p}}} \left( \frac{1}{2\pi}\int_{\mathbb{T}} |f(\phi_t(\zeta))|^p |w_t(\zeta)|^{p-2} |\phi_t'(\zeta)| |d\zeta| \right)^{\frac{1}{p}} \notag \\
&= \frac{t^{\frac{1}{p}-1}}{(1-t)^{\frac{1}{p}-b}} \left(\frac{1}{2\pi} \int_{\phi_t(\T)} |f(w)|^p |w|^{p-2} |dz| \right)^{\frac{1}{p}}.
\end{align*}
Here the last equality follows from the fact that
\[
w_t(\phi_t^{-1}(z)) = \frac{z}{t} (1 - t)^b.
\]
 The proof is complete.
\end{proof}

The second ingredient is Gabriel's classical theorem on integrals of the moduli of analytic functions along circles inside the unit disc.  We formulate it in the form used below; see Gabriel \cite{Gabriel1928} or Frazer \cite{Frazer1934} for the sharp constant $1$ when the curve is a circle.

\begin{lemma}[Gabriel]\label{lem:gabriel}
Let $f\in H^p$, $1\le p<\infty$, and let $\gamma$ be a circle contained in $\D$.  Then
\[
\frac{1}{2\pi}\int_\gamma |f(\zeta)|^p\,|d\zeta|
\le \frac{1}{2\pi}\int_\T |f(\zeta)|^p\,|d\zeta|
=\|f\|_{H^p}^p.
\]
\end{lemma}

Strictly speaking, Gabriel's theorem is usually stated for circles strictly inside $\D$; the case of the tangent circle $\phi_t(\T)$ follows by applying the theorem to the slightly smaller circles $\phi_t(r\T)$, $r\rightarrow1$, and using the $H^p$ boundary representation \eqref{eq:boundary}.

\section{ Lower bound for $1\leq p<\infty$  }
\label{sec:lower}

We now establish the matching lower bound of $\|\mathcal{H}_b\|_{H^p\to H^p}$ for every $1\leq p<\infty$.

\begin{theorem}\label{thm:lower}
Let $b>0$ and $1\leq p<\infty$.  Then
\[
\|\mathcal{H}_b\|_{H^p\to H^p}\ge B\Bigl(\frac1p,b+1-\frac1p\Bigr).
\]
\end{theorem}

\begin{proof}
For $0<\gamma<1$ consider the test functions
\begin{equation}\label{eq:test}
f_\gamma(z)=(1-z)^{-\gamma/p},\qquad z\in\D.
\end{equation}
A classical fact about power kernels is that $f_\gamma\in H^p$ for every $\gamma<1$ and that $\|f_\gamma\|_{H^p}\to\infty$ as $\gamma\to1^-$; see Duren \cite[Chapter~5]{Duren70}.

Using the integral representation from (\ref{eq:Hb}), together with (\ref{eq:Euler}), we have
\[
\mathcal{H}_bf_\gamma(z) = \int_0^1 \frac{(1-t)^{b-\gamma/p}}{(1-tz)^{b+1}}d  t
= B\!\left(1, b+1-\frac{\gamma}{p}\right)
  \,F\!\left(1, b+1, b+2-\frac{\gamma}{p}; z\right).
\]
The hypergeometric series expansion then takes the form
\[
\mathcal{H}_bf_\gamma(z)
= \frac{\Gamma\!\left(\frac{\gamma}{p}\right)\Gamma\!\left(b+1-\frac{\gamma}{p}\right)}{\Gamma\!\left(b+1\right)}
  \sum_{k=0}^{\infty}
  \frac{\Gamma(k+1)\Gamma(k+b+1)}
       {\Gamma\!\left(k+b+2-\frac{\gamma}{p}\right)
        \Gamma\!\left(k+\frac{\gamma}{p}\right)}
  \frac{\Gamma\!\left(k+\frac{\gamma}{p}\right)}
       {\Gamma\!\left(\frac{\gamma}{p}\right)}
  \frac{z^k}{k!}.
\]
An application of Stirling's formula yields
\[
\frac{\Gamma(k+1)\Gamma(k+b+1)}
     {\Gamma\!\left(k+b+2-\frac{\gamma}{p}\right)
      \Gamma\!\left(k+\frac{\gamma}{p}\right)}
= 1 + O\!\left(\frac{1}{k+1}\right), \qquad k\to\infty.
\]
Consequently,

\begin{align}
\mathcal{H}_b(f_\gamma)(z)
&= B\!\left(\frac{\gamma}{p},\,b+1-\frac{\gamma}{p}\right)
   \sum_{k=0}^{\infty}
   \left[1+O\!\left(\frac{1}{k+1}\right)\right]
   \frac{\Gamma\!\left(k+\frac{\gamma}{p}\right)}
        {\Gamma\!\left(\frac{\gamma}{p}\right)}
   \frac{ z^k}{k!} \nonumber \\
&= B\!\left(\frac{\gamma}{p},\,b+1-\frac{\gamma}{p}\right)
   \left(f_\gamma(z) + g_\gamma(z)\right), \label{eq:Hb_fgamma}
\end{align}

where the error term $g_\gamma$ satisfies, since $\gamma<p$,
\[
\sup_{0<\gamma<1}\|g_\gamma\|_{H^\infty} \le C < \infty.
\]
Finally, since \(H^\infty\subset H^p\) on the unit disk, we obtain
\[
\|g_\gamma\|_{H^p} \le \|g_\gamma\|_{H^\infty} <\infty.
\]
Therefore, taking the norm on both sides of \eqref{eq:Hb_fgamma} and then passing to the limit as $\gamma \to 1^-$, we establish the lower bound

\[
\begin{aligned}
\|\mathcal{H}_b\|_{H^p \to H^p}
&\geq \lim_{\gamma \to 1^-} \left[B\!\left(\frac{\gamma}{p},\,b+1-\frac{\gamma}{p}\right)
\frac{\|f_\gamma\|_{H^p} - \|g_\gamma\|_{H^p}}{\|f_\gamma\|_{H^p}} \right]\\
&=B\!\left(\frac{1}{p},\,b+1-\frac{1}{p}\right).
\end{aligned}
\]
This completes the proof.
\end{proof}

\section{ Upper bound for $1< p<\infty$ }
\label{sec:upper}

In this section we first prove the upper bound
\begin{equation}\label{eq:upper}
\|\mathcal{H}_b\|_{H^p\to H^p}\le B\Bigl(\frac1p,b+1-\frac1p\Bigr),
\qquad 2\le p<\infty.
\end{equation}
When $p\ge 2$, this upper bound is usually the easiest to obtain. Although the result has already been established in \cite{Li2009}, we provide a different proof here for the sake of completeness.

\begin{proposition}\label{prop:upper}
Let $b>0$ and $2\leq p<\infty$.  Then \eqref{eq:upper} holds.
\end{proposition}

\begin{proof}
Using the weighted composition representation \eqref{eq:WCO} and Minkowski's integral inequality, we have
\begin{align}\label{eqhb}
\|\mathcal{H}_b f\|_{H^p}
\le\int_0^1\,\|T_t f\|_{H^p}\,dt.
\end{align}
By Lemma \ref{lem:wco}, we deduce that
\[
\|T_t f\|_{H^p}^p
=\frac{t^{1-p}}{(1-t)^{1-bp}} \frac{1}{2\pi} \int_{\phi_t(\T)} |f(w)|^p |w|^{p-2} |dw| .
\]
Since $p\ge2$ and $|w|\le1$ on $\phi_t(\T)$, we have $|w|^{p-2}\le1$.  Lemma \ref{lem:gabriel} then gives
\[
\frac{1}{2\pi}\int_{\phi_t(\T)}|f(w)|^p|w|^{p-2}\,|dw|
\le\frac{1}{2\pi}\int_{\phi_t(\T)}|f(w)|^p\,|dw|
\le\|f\|_{H^p}^p.
\]
Therefore

\begin{align}\label{eqtt}
\|T_t f\|_{H^p}\le \frac{t^{\frac{1}{p}-1}}{(1-t)^{\frac{1}{p}-b}} \|f\|_{H^p},
\qquad 0<t<1.
\end{align}
Combining \eqref{eqhb} and \eqref{eqtt} therefore gives
\[
\|\mathcal{H}_b f\|_{H^p}
\le\|f\|_{H^p}\int_0^1 t^{1/p-1}(1-t)^{b-1/p}\,dt
=B\Bigl(\frac1p,b+1-\frac1p\Bigr)\|f\|_{H^p}.
\]
This proves \eqref{eq:upper}.
\end{proof}

It remains to consider the case $1 < p < 2$. Let $q$ be the conjugate exponent, so that $1/p+1/q=1$ and $q>2$.  Since $(H^p)^*\cong H^q$ under the pairing \eqref{eq:pairing}, we have
\begin{equation}\label{eq:duality}
\|\mathcal{H}_b\|_{H^p\to H^p}=\|\mathcal{H}_b^*\|_{H^q\to H^q}.
\end{equation}

The adjoint is given explicitly by the following lemma; it follows from the same Cauchy-pairing computation as in \cite[Prop.~1]{Bellavita2024}.

\begin{lemma}\label{lem:adjoint}
For $1<p<\infty$ the adjoint $\mathcal{H}_b^*:H^q\to H^q$ is represented by
\begin{equation}\label{eq:adjoint}
\mathcal{H}_b^* g(z)=\int_0^1\frac{t^b}{1-(1-t)z}\,g(\phi_t(z))\,dt,
\qquad z\in\D,
\end{equation}
for every $g\in H^q.$
\end{lemma}

\begin{proposition}\label{prop:dual}
Let $b>0$ and $1<p<2$.  Then
\begin{equation} \label{eq:dual-upper}
\|\mathcal{H}_b\|_{H^p\to H^p}\le B\Bigl(\frac1p,\,b+1-\frac1p\Bigr).
\end{equation}
\end{proposition}

\begin{proof}
Since $q>2$, we may apply the same estimate as in Proposition \ref{prop:upper} and Lemma \ref{lem:adjoint}.
\begin{align*}
\|\mathcal{H}_b^*(f)\|_{H^q}
&\le\int_0^1 t^{1/q-1}(1-t)^{-1/q}\,t^b\,d t\|f\|_{H^q}=\int_0^1 t^{b+1/q-1}(1-t)^{-1/q}\,d t\|f\|_{H^q}\\
&= B\!\left(\frac{1}{p},\,b+1-\frac{1}{p}\right)\|f\|_{H^q}.
\end{align*}
The last equality follows from the fact that $1/q = 1 - 1/p$ and the symmetry of the Beta function:
\[
B\Bigl(b+\frac1q,1-\frac1q\Bigr)
=B\Bigl(b+1-\frac1p,\frac1p\Bigr)
=B\Bigl(\frac1p,b+1-\frac1p\Bigr).
\]
Thus \eqref{eq:dual-upper} follows from \eqref{eq:duality}.
\end{proof}

\begin{proof}[{\bf Proof of Theorem {\rm\ref{thm:main}}}. ] Combining Propositions \ref{prop:upper} and \ref{prop:dual} with Theorem \ref{thm:lower} yields Theorem \ref{thm:main}.
\end{proof}

\section{\texorpdfstring{Boundedness on $H^1$ and unboundedness on $H^\infty$}{Boundedness on H1 and unboundedness on H-infinity}}
\label{sec:endpoints}
It is well known that the classical Hilbert operator $\mathcal H$ fails to be bounded on $H^1$ \cite{DiamantopoulosSiskakis2000}. In the case of the generalized Hilbert operator $\mathcal H_b$, however, the situation is completely different. In this section we prove Theorem \ref{thm:endpoints}.

\begin{proof}[{\bf Proof of Theorem \ref{thm:endpoints}}] ${\bf (i).}$
Let $f\in H^1$. Applying Fubini's theorem together with the boundary representation \eqref{eq:boundary} yields
\[
\|\mathcal{H}_b f\|_{H^1}
\le \frac{1}{2\pi}\int_0^{2\pi}\int_0^1 |f(t)|\,\frac{(1-t)^b}{|1-te^{i\theta}|^{b+1}}\,dt\,d\theta
=\int_0^1 |f(t)|\,(1-t)^b\,K_b(t)\,dt,
\]
where
\[
K_b(t)=\frac{1}{2\pi}\int_0^{2\pi}\frac{d\theta}{|1-te^{i\theta}|^{b+1}}.
\]
 By expanding $(1-ae^{i\theta})^{-(b+1)/2}$ and $(1-ae^{-i\theta})^{-(b+1)/2}$ in binomial series, multiplying, and integrating termwise,
 we have
\[
K_b(t)=F\Bigl(\frac{b+1}{2},\frac{b+1}{2},1;t^2\Bigr).
\]
For $b>0$ this hypergeometric function has a single singularity at $t=1$ and satisfies $$K_b(t)\le C_b(1-t)^{-b}$$ for some constant $C_b>0$. This follows from the standard asymptotic of the hypergeometric function at its singular point $z=1$, or directly from the integral estimate above.  Hence $(1-t)^bK_b(t)\le C_b$ for all $t\in[0,1)$.  Therefore
\[
\|\mathcal{H}_b f\|_{H^1}\le C_b\int_0^1|f(t)|\,dt
\le C_b C\|f\|_{H^1}
\]
by the Fej\'er--Riesz inequality \eqref{eq:FR}.  Thus $\mathcal{H}_b$ is bounded on $H^1$.

On the other hand, if $f \in H^1$ and $f(0)=0$, then $f$ admits a decomposition
\begin{equation}\label{eq-fz}
f(z)=zg(z), \qquad g\in H^1,
\end{equation}
with $\|f\|_{H^1}=\|g\|_{H^1}$.
Let \(S\) denote the unilateral shift on \(H^1\), i.e.,  \(Sg(z)=zg(z)\).
Applying \(T_t\) to both sides and then taking the \(H^1\) norm yields
\begin{align}\label{eq-t1}
\|T_t f\|_{H^1}=\|T_t (S g)\|_{H^1}.
\end{align}
  Since
\[
\phi_t'(z) = \frac{1-t}{t} \cdot (\phi_t(z))^2,
\]
by Lemma \ref{lem:gabriel} and a change of variables, we obtain
\begin{align}
\|T_t(Sg)\|_{H^1}
&= \frac{1}{2\pi}\int_{\mathbb{T}} |g(\phi_t(\zeta))| |\phi_t(\zeta)||w_t(\zeta)| |d\zeta|  \notag \\
&= \frac{(1-t)^b}{t}\frac{1}{2\pi}\int_{\mathbb{T}} |g(\phi_t(\zeta))| |\phi_t(\zeta)|^2 |d\zeta|  \notag \\
&= (1-t)^{b-1} \frac{1}{2\pi}\int_{\phi_t(\mathbb{T})} |g(w)|\,|dw| \notag \\
&\leq  (1-t)^{b-1} \|g\|_{H^1}=(1-t)^{b-1} \|f\|_{H^1}. \label{eq:Tt_Sg_bound}
\end{align}
Therefore, \eqref{eq:main1} is a direct consequence of \eqref{eqhb}, \eqref{eq:Tt_Sg_bound}, and Theorem \ref{thm:lower}.

${\bf (ii)}$.  Apply $\mathcal{H}_b$ to the constant function $f\equiv1$.  For $0<r<1$ we have
\[
\mathcal{H}_b(1)(r)=\int_0^1\frac{(1-t)^b}{(1-tr)^{b+1}}\,dt.
\]
With $u=1-t$ this becomes
\[
\mathcal{H}_b(1)(r)=\int_0^1\frac{u^b}{(1-r+ru)^{b+1}}\,du.
\]
Since $1-r+ru\le 1-r+u$ for $0\le u\le1$, we obtain
\[
\mathcal{H}_b(1)(r)\ge\int_0^1\frac{u^b}{(1-r+u)^{b+1}}\,du.
\]
For $1-r\le u\le1$ we have $1-r+u\le2u$, so
\[
\frac{u^b}{(1-r+u)^{b+1}}\ge\frac{1}{2^{b+1}u}.
\]
Consequently
\[
\mathcal{H}_b(1)(r)\ge\frac{1}{2^{b+1}}\int_{1-r}^1\frac{du}{u}
=\frac{1}{2^{b+1}}\log\frac{1}{1-r},
\]
which tends to $+\infty$ as $r\to1^-$.  Hence $\mathcal{H}_b(1)\notin H^\infty$, and $\mathcal{H}_b$ is unbounded on $H^\infty$.
\end{proof}

\begin{remark}
We provide a detailed comparison of $\mathcal{H}_{b}$ with the classical
Hilbert operator $\mathcal{H}=\mathcal{H}_{0}$.  The most significant qualitative difference is the boundedness
on $H^{1}$ for $b>0$.    For $b=0$ and $p\to 1^{+}$, $B(1/p,1-1/p)\to\infty$,
reflecting the unboundedness of $\mathcal{H}$ on $H^{1}$.
For $b>0$ and $p\to 1^{+}$, $$B\Bigl(\frac{1}{p},b+1-\frac{1}{p}\Bigr) \to B(1,b)=\frac{1}{b},$$
which is finite, consistent with Theorem~\ref{thm:endpoints}~(i).
For $p\to\infty$, the Hardy space $H^{p}$ approaches $H^{\infty}$
in a suitable sense.  Our norm formula gives
\begin{align*}
\lim_{p\to\infty} B\Bigl(\frac{1}{p},b+1-\frac{1}{p}\Bigr)
&= \lim_{p\to\infty}
\frac{\Gamma(1/p)\Gamma(b+1-\frac{1}{p})}{\Gamma(b+1)}
= \lim_{p\to\infty}\Gamma(1/p) = \infty,
\end{align*}
This divergence is consistent with Theorem~\ref{thm:endpoints}~(ii).
\end{remark}

\section{Proof of Theorem \ref{thm:bloch}}
\label{sec:bloch}

Theorem \ref{thm:endpoints} shows that $\mathcal{H}_b:H^\infty\to H^\infty$ is unbounded. However, if the target space is enlarged to the Bloch space $\mathcal B$, boundedness is restored. In this section we prove Theorem \ref{thm:bloch}, in which we determine the exact norm of $\mathcal{H}_b:H^\infty\to\mathcal B$.

Recall that
\[
\|\mathcal{H}_b f\|_\B=|\mathcal{H}_b f(0)|+\sup_{z\in\D}(1-|z|^2)|(\mathcal{H}_b f)'(z)|.
\]

\begin{proof}[{\bf Proof of Theorem \ref{thm:bloch}} ]
 {\bf Upper bound}.
Let $f\in H^\infty$ with $\|f\|_{H^\infty}\le1$.  From \eqref{eq:Hb} we have
\begin{equation}\label{eq:Hb0}
|\mathcal{H}_b f(0)|
\le\int_0^1(1-t)^b\,dt=\frac{1}{b+1}.
\end{equation}
Differentiating under the integral sign gives
\begin{equation}\label{eq:Hbp}
(\mathcal{H}_b f)'(z)=(b+1)\int_0^1 t\,f(t)\,\frac{(1-t)^b}{(1-tz)^{b+2}}\,dt.
\end{equation}
For $|z|=r<1$ this implies
\[
(1-r^2)|(\mathcal{H}_b f)'(z)|
\le (b+1)(1-r^2)\int_0^1\frac{t(1-t)^b}{(1-tr)^{b+2}}\,dt
=:I_b(r).
\]
Set
\[
J_b(r):=\int_0^1\frac{t(1-t)^b}{(1-tr)^{b+2}}\,dt,
\qquad 0\le r<1.
\]
With the substitution $u=(1-t)/(1-tr)$ we obtain $t=(1-u)/(1-ru)$, $1-t=u(1-r)/(1-ru)$, $1-tr=(1-r)/(1-ru)$, and
\[
J_b(r)=\frac{1}{1-r}\int_0^1\frac{(1-u)u^b}{1-ru}\,du.
\]
Consequently
\begin{equation}\label{eq:Ib}
I_b(r)=(b+1)(1+r)\int_0^1\frac{(1-u)u^b}{1-ru}\,du.
\end{equation}
Differentiation under the integral sign is justified because the integrand is non-negative and smooth, and from \eqref{eq:Ib} we obtain
\[
I_b'(r)=(b+1)\int_0^1\frac{(1-u)u^b(1+u)}{(1-ru)^2}\,du\ge0,
\]
so $I_b$ is increasing on $[0,1)$.  Letting $r\to1^-$ in \eqref{eq:Ib} and using Dominated Convergence Theorem gives
\[
\lim_{r\to1^-}I_b(r)
=(b+1)\cdot2\int_0^1 u^b\,du
=(b+1)\cdot2\cdot\frac{1}{b+1}=2.
\]
Hence $\sup\limits_{0\le r<1}I_b(r)=2$.  Together with \eqref{eq:Hb0} this yields
\[
\|\mathcal{H}_b f\|_\B\le\frac{1}{b+1}+2.
\]

{\bf Lower bound}.
Take the constant function $f\equiv1$, which satisfies $\|f\|_{H^\infty}=1$.  Then $\mathcal{H}_b(1)(0)=1/(b+1)$ and $(\mathcal{H}_b(1))'(z)$ is exactly the derivative computed in \eqref{eq:Hbp} with $f\equiv1$.  Therefore
\[
\|\mathcal{H}_b(1)\|_\B
\ge\frac{1}{b+1}+\sup_{0\le r<1}I_b(r)
=\frac{1}{b+1}+2.
\]
This matches the upper bound and completes the proof of Theorem \ref{thm:bloch}.

\begin{remark} Setting $b=0$ in Theorem~\ref{thm:bloch} gives
$\norm{\mathcal{H}}_{H^{\infty}\to\mathcal{B}}=3$, recovering the result
of~\cite{HY2025}.
\end{remark}

\end{proof}

{\bf Data Availability}  No data was used to support this study. \vskip 2mm
	
{\bf Conflicts of Interest}  The authors  declare that they have no conflicts of interest. \vskip 2mm

{\bf Acknowledgements} 	The first author was supported by  NNSF of China (No. 12371131). The  corresponding author was supported by Sichuan Provincial Natural Science Foundation of China (No. 2026NSFSC0718) and Guizhou Education Department Youth Science and Technology Talent Growth Project (No. QianJiaoJi [2024] 159). 


\vskip 2mm


\begin{thebibliography}{10}

\bibitem{bw} G. Bao and  H. Wulan,  Hankel matrices acting on Dirichlet spaces, \textit{ J. Math. Anal.   Appl.}  {\bf 409} (2014), no.~1, 228--235.

\bibitem{Bellavita2024} C. Bellavita, N. Chalmoukis, V. Daskalogiannis and G. Stylogiannis, Generalized Hilbert operators arising from Hausdorff matrices,
\textit{Proc. Amer. Math. Soc.} \textbf{152} (2024), no.~11, 4759--4773.

\bibitem{DiamantopoulosSiskakis2000} E. Diamantopoulos and A. G. Siskakis, Composition operators and the Hilbert matrix,
\textit{Studia Math.} \textbf{140} (2000), no.~2, 191--198.

\bibitem{DostanicJevticVukotic2008} M. Dostani\'c, M. Jevti\'c and D. Vukoti\'c, Norm of the Hilbert matrix on Bergman and Hardy spaces and a theorem of Nehari type,
\textit{J. Funct. Anal.} \textbf{254} (2008), no.~11, 2800--2815.

\bibitem{Duren70} P.  Duren, Theory of $H^p$ spaces, 
\textit{Academic Press, New York,} 1970; \textit{reprinted by Dover, Mineola, NY,} 2000.

\bibitem{Frazer1934} H. Frazer, On the moduli of regular functions, \textit{Proc. Lond. Math. Soc.}  \textbf{37} (1934), no.~2, 49--65.

\bibitem{Gabriel1928} R.   Gabriel, Some results concerning the integrals of moduli of regular functions along curves of certain types,
\textit{Proc. Lond. Math. Soc.}   \textbf{28} (1928), no.~2, 121--127.



\bibitem{HY2025} H.~Hu and S.~Ye, Norm of the Hilbert matrix operator between some spaces of analytic functions,
\emph{J.~Geom.\ Anal.}\ \textbf{35} (2025), 184.

\bibitem{Li2009} S.~Li, Generalized Hilbert operator on Dirichlet-type space, \emph{Appl.\ Math.\ Comput.}\ \textbf{214} (2009), no.~1, 304--309.

\bibitem{LiStevic2009} S.~Li and S.~Stevi\'c, Generalized Hilbert operator and Fej\'er--Riesz type inequalities on the polydisc,
\emph{Acta Math.\ Sci.\ Ser.\ B} \textbf{29} (2009), no.~1, 191--200.

\bibitem{zkh} K. Zhu,  Operator Theory in Function Spaces,
\textit{   Second Edition, Amer. Math. Soc. } 2007.

\end{thebibliography}
\end{document}